\documentclass[12pt, reqno]{amsart}

\author[P.~Leonetti]{Paolo Leonetti$^\dagger\,$}
\thanks{$\dagger\,$P.~Leonetti is supported by the Austrian Science Fund (FWF), project F5512-N26}
\address{Graz University of Technology, Institute of Analysis and Number Theory, Graz, Austria}
\email{leonetti.paolo@gmail.com}

\author[C.~Sanna]{Carlo Sanna$^\ddagger$}
\thanks{$\ddagger\,$C.~Sanna is supported by a postdoctoral fellowship of INdAM and is a member of the INdAM group GNSAGA}
\address{Universit\`a degli Studi di Genova\\Department of Mathematics\\Genova, Italy}
\email{carlo.sanna.dev@gmail.com}

\title{Practical numbers among the\\ binomial coefficients}

\usepackage{amsmath}
\usepackage{amssymb}
\usepackage{amsthm}
\usepackage{nicefrac}
\usepackage[left=3.5cm, right=3.5cm]{geometry}
\usepackage{color}
\usepackage{fancyhdr}
\usepackage{hyperref}
\usepackage{enumerate}

\hypersetup{
    pdftitle={},
    pdfauthor={},
    pdfmenubar=false,
    pdffitwindow=true,
    pdfstartview=FitH,
    colorlinks=true,
    linkcolor=blue,
    citecolor=green,
    urlcolor=cyan
}

\newtheorem{thm}{Theorem}[section]
\newtheorem{cor}{Corollary}[section]
\newtheorem{lem}[thm]{Lemma}
\newtheorem{pro}{Proposition}[section]
\theoremstyle{remark}
\newtheorem{rmk}{Remark}[section]
\newtheorem{que}{Question}[section]

\begin{document}

\begin{abstract}
A \emph{practical number} is a positive integer $n$ such that every positive integer less than $n$ can be written as a sum of distinct divisors of $n$.
We prove that most of the binomial coefficients are practical numbers.
Precisely, letting $f(n)$ denote the number of binomial coefficients $\binom{n}{k}$, with $0 \leq k \leq n$, that are not practical numbers, we show that
\begin{equation*}
f(n) < n^{1 - (\log 2 - \delta)/\log \log n}
\end{equation*}
for all integers $n \in [3, x]$, but at most $O_\gamma(x^{1 - (\delta - \gamma) / \log \log x})$ exceptions, for all $x \geq 3$ and $0 < \gamma < \delta < \log 2$.
Furthermore, we prove that the central binomial coefficient $\binom{2n}{n}$ is a practical number for all positive integers $n \leq x$ but at most $O(x^{0.88097})$ exceptions.
We also pose some questions on this topic.
\end{abstract}

\maketitle
\thispagestyle{empty}

\section{Introduction}

A \emph{practical number} is a positive integer $n$ such that every positive integer less than $n$ can be written as a sum of distinct divisors of $n$.
This property has been introduced by Srinivasan~\cite{MR0027799}.
Estimates for the counting function of practical numbers have been given by Hausman--Shapiro~\cite{MR752596}, Tenenbaum~\cite{MR860809}, Margenstern~\cite{MR1089787}, Saias~\cite{MR1430008}, and finally Weingartner~\cite{MR3356847}, who proved that there are asymptotically $Cx/\log x$ practical numbers less than $x$, for some constant $C>0$, as previously conjectured by Margenstern~\cite{MR1089787}. 
On another direction, Melfi~\cite{MR1370203} proved that every positive even integer is the sum of two practical numbers, and that there are infinitely many triples $(n,n+2,n+4)$ of practical numbers. 
Also, Melfi~\cite{MR1452391} proved that in every Lucas sequence, satisfying some mild conditions, there are infinitely many practical numbers, and Sanna~\cite{SannaInPress} gave a lower bound for their counting function.

In this work, we study the binomial coefficients which are also practical numbers. 
Our first result, informally, states that for almost all positive integers $n$ there is a negligible amount of binomial coefficients $\binom{n}{k}$, with $0 \leq k \leq n$, which are not practical. 
Precisely, for each positive integer $n$, define 
\begin{equation*}
f(n) := \#\!\left\{0 \leq k \leq n : \binom{n}{k} \text{ is not a practical number}\right\}.
\end{equation*}
Our first result is the following.

\begin{thm}\label{thm:row}
For all $x \geq 3$ and $0 < \gamma < \delta < \log 2$, we have 
\begin{equation*}
f(n) < n^{1 - (\log 2 - \delta)/\log \log n}
\end{equation*}
for all integers $n \in [3, x]$, but at most $O_\gamma(x^{1 - (\delta - \gamma) / \log \log x})$ exceptions.
\end{thm}

As a consequence, we obtain that as $x \to +\infty$ almost all binomial coefficients $\binom{n}{k}$, with $0 \leq k \leq n \leq x$, are practical numbers.

\begin{cor}\label{cor:nk}
We have
\begin{equation*}
\sum_{n \,\leq\, x} f(n) \ll_\varepsilon x^{2 - (\frac1{2}\log 2 - \varepsilon) / \log \log x} ,
\end{equation*}
for all $x \geq 3$ and $\varepsilon > 0$.
\end{cor}

Among the binomial coefficients, the \emph{central binomial coefficients} $\binom{2n}{n}$ are of great interest.
In particular, several authors have studied their arithmetic and divisibility properties, see e.g.~\cite{MR1613294, MR0369288, MR3383891, MR3801088, MR777971}.

In this direction, our second result, again informally, states that almost all central binomial coefficients $\binom{2n}{n}$ are practical numbers.

\begin{thm}\label{thm:central}
For $x \geq 1$, the central binomial coefficient $\binom{2n}{n}$ is a practical number for all positive integers $n \leq x$ but at most $O(x^{0.88097})$ exceptions.
\end{thm}

Probably, there are only finitely many positive integers $n$ such that $\binom{2n}{n}$ is not a practical number.
By a computer search, we found only three of them below $10^6$, namely $n = 4, 10, 256$.
However, proving the finiteness seems to be out of reach with actual techniques. 
Indeed, on the one hand, if $n$ is a power of $2$ whose base $3$ representation contains only the digits $0$ and $1$, then it can be shown that $\binom{2n}{n}$ is not a practical number (see Proposition~\ref{pro:23} below).
On the other hand, it is an open problem to establish whether there are finitely or infinitely many powers of $2$ of this type~\cite{MR580438, MR1845703, MR2506687, MR623247}.

We conclude by leaving two open questions.
Note that since $\binom{n}{0} = \binom{n}{n} = 1$, we have $0 \leq f(n) \leq n - 1$ for all positive integers $n$.
It is natural to ask when one of the equalities is satisfied.

\begin{que}\label{que:f2}
What are the positive integers $n$ such that $f(n) = 0$ ?
\end{que}

\begin{que}\label{que:fn1}
What are the positive integers $n$ such that $f(n) = n - 1$ ?
\end{que}

Regarding Question~\ref{que:f2}, if $f(n) = 0$ then $n$ must be a power of $2$, otherwise there would exist (see Lemma~\ref{lem:Fine} below) an odd binomial coefficient $\binom{n}{k}$, with $0 < k < n$, and since $1$ is the only odd practical number, we would have $f(n) > 0$.
However, this is not a sufficient condition, since $f(8) = 1$.
Regarding Question~\ref{que:fn1}, if $n = 2^k - 1$, for some positive integer $k$, then $f(n) = n - 1$, because all the binomial coefficients $\binom{n}{k}$, with $0 < k < n$, are odd (see Lemma~\ref{lem:Fine} below) and greater than $1$, and consequently they are not practical numbers.
However, this is not a necessary condition, since $f(5) = 4$.

\subsection*{Notation}
We employ the Landau--Bachmann ``Big Oh'' notation $O$ and the associated Vinogradov symbol $\ll$.
In particular, any dependence of the implied constants is indicated with subscripts.
We write $p_i$ for the $i$th prime number.

\section{Preliminaries}

This section is devoted to some preliminary results needed in the later proofs.
We begin with some lemmas about practical numbers.

\begin{lem}\label{lem:practicaldouble}
If $m$ is a practical number and $n \leq 2m$ is a positive integer, then $mn$ is a practical number.
\end{lem}
\begin{proof}
See~\cite[Lemma~4]{MR1452391}.
\end{proof}

\begin{lem}\label{lem:practical2d}
If $d$ is a practical number and $n$ is a positive integer divisible by $d$ and having all prime factors not exceeding $2d$, then $n$ is a practical number.
\end{lem}
\begin{proof}
By hypothesis, there exist positive integers $q_1,\ldots,q_k\le 2d$ such that $n=dq_1\cdots q_k$. 
Then, using Lemma~\ref{lem:practicaldouble}, it follows by induction that $dq_1\cdots q_m$ is practical for all $m=1,\ldots,k$. In particular, $n$ is practical.
\end{proof}

\begin{lem}\label{lem:practicalp1ps}
We have that $p_1^{a_1} \cdots p_s^{a_s}$ is a practical number, for all positive integers $a_1, \dots, a_s$.
\end{lem}
\begin{proof}
It follows easily by induction on $s$, using Lemma~\ref{lem:practicaldouble} and Bertrand's postulate $p_{i + 1} < 2p_i$.
\end{proof}

For each prime number $p$ and for each positive integer $n$, put
\begin{equation*}
T_p(n) := \#\!\left\{0 \leq k \leq n : p \nmid \binom{n}{k} \right\} .
\end{equation*}
We have the following formula for $T_p(n)$.

\begin{lem}\label{lem:Fine}
Let $p$ be a prime number and let
\begin{equation*}
n = \sum_{j \,=\, 0}^s d_j p^j, \quad d_0,\dots,d_s \in \{0,\dots, p - 1\}, \quad d_s \neq 0,
\end{equation*}
be the representation in base $p$ of the positive integer $n$.
Then we have
\begin{equation*}
T_p(n) = \prod_{j \,=\, 0}^s (d_j + 1) .
\end{equation*}
\end{lem}
\begin{proof}
See~\cite[Theorem~2]{MR0023257}.
\end{proof}

For each prime number $p$, let us define
\begin{equation*}
\omega_p := \frac{\log((p + 1) / 2)}{\log p} .
\end{equation*}
The quantity $\omega_p$ appears in the following upper bound for $T_p(n)$.

\begin{lem}\label{lem:Tpupper}
Let $p$ be a prime number and fix $\varepsilon \in (0,\nicefrac{1}{2})$.
Then, for all $x \geq 1$, we have 
\begin{equation*}
T_p(n) < n^{\omega_p + \varepsilon}
\end{equation*}
for all positive integers $n \leq x$ but at most $O(p^3 x^{1 - \varepsilon})$ exceptions.
\end{lem}
\begin{proof}
For $x\ge 1$, let $k$ be the smallest integer such that $x < p^k$.
Clearly, we have
\begin{align}\label{equ:boundTper}
E(x) &:= \#\big\{n \leq x : T_p(n) \geq n^{\omega_p + \varepsilon}\big\} \nonumber\\
&\leq \sum_{j \,=\, 1}^k \#\big\{p^{j - 1} \leq n < p^j : T_p(n) \geq p^{(j - 1)(\omega_p + \varepsilon)}\big\} .
\end{align}
Moreover, thanks to Lemma \ref{lem:Fine}, we have
\begin{equation*}
\sum_{p^{j-1} \,<\, n \,\leq\, p^j} T_p(n) \leq \sum_{0 \,\leq\, d_0, \dots, d_{j-1} \,<\, p} \prod_{i\,=\,0}^{j-1} (d_i + 1) = \left(\sum_{d \,=\, 0}^{p - 1} (d + 1)\right)^j = \left(\frac{p(p+1)}{2}\right)^j ,
\end{equation*}
and consequently
\begin{align}\label{equ:boundTj}
\#\big\{p^{j - 1} &\leq n < p^j : T_p(n) \geq p^{(j - 1)(\omega_p + \varepsilon)}\big\} \leq \frac1{p^{(j - 1)(\omega_p + \varepsilon)}} \sum_{p^{j-1} \,<\, n \,\leq\, p^j} T_p(n) \nonumber\\
&\leq \frac1{p^{(j - 1)(\omega_p + \varepsilon)}} \left(\frac{p(p+1)}{2}\right)^j = \frac{p(p+1)}{2} \, p^{(1 - \varepsilon)(j - 1)} < p^{2 + (1 - \varepsilon)(j - 1)} ,
\end{align}
for all positive integers $j$.
Therefore, putting together \eqref{equ:boundTper} and \eqref{equ:boundTj}, and using that $\varepsilon < \nicefrac1{2}$, we obtain
\begin{equation}\label{equ:Ell}
E(x) < \sum_{j \,=\, 1}^k p^{2 + (1 - \varepsilon)(j - 1)} \ll p^{2 + (1 - \varepsilon)k} \leq p^{2 + (1 - \varepsilon)(\log x / \log p + 1)} < p^3 x^{1 - \varepsilon} ,
\end{equation}
as desired.
\end{proof}

\begin{rmk}
The constant $\nicefrac1{2}$ in the statement of Lemma~\ref{lem:Tpupper} has no particular importance, it is only needed to justify the $\ll$ in~\eqref{equ:Ell}.
Any other real number less than $1$ would be fine.
\end{rmk}

For all $x \geq 1$, let $\kappa(x)$ be the smallest integer $k \geq 1$ such that $p_1 \cdots p_k \geq x$.

\begin{lem}\label{lem:asymptoticssxandpsx}
We have
\begin{equation*}
\kappa(x) \sim \frac{\log x}{\log \log x} \quad\text{and}\quad p_{\kappa(x)} \sim \log x ,
\end{equation*}
as $x\to \infty$.
\end{lem}
\begin{proof}
As a well-known consequence of the Prime Number Theorem, we have
\begin{equation}\label{equ:logprimorial}
\log(p_1 \cdots p_k) \sim p_k \sim k \log k ,
\end{equation}
as $k \to +\infty$.
Since
\begin{equation*}
\log(p_1 \cdots p_{\kappa(x) - 1}) < \log x \leq \log(p_1 \cdots p_{\kappa(x)}) ,
\end{equation*}
and $\kappa(x) \to +\infty$ as $x \to +\infty$, by \eqref{equ:logprimorial} we obtain
\begin{equation*}
p_{\kappa(x)} \sim \kappa(x) \log \kappa(x) \sim \log x ,
\end{equation*}
which in turn implies
\begin{equation*}
\kappa(x) \sim \frac{\kappa(x) \log \kappa(x)}{\log \kappa(x) + \log\log \kappa(x)} \sim \frac{\log x}{\log \log x} ,
\end{equation*}
as desired.
\end{proof}

For every prime number $p$ and every positive integer $n$, let $\beta_p(n)$ be the $p$-adic valuation of the central binomial coefficient $\binom{2n}{n}$.

\begin{lem}\label{lem:betap}
For each prime $p$ and all positive integers $n$, we have that $\beta_p(n)$ is equal to the number of digits of $n$ in base $p$ which are greater than $(p - 1) / 2$.
\end{lem}
\begin{proof}
The claim is a straightforward consequence of a theorem of Kummer~\cite{MR1578793} which says that, for positive integers $m, n$, the $p$-adic valuation of $\binom{m + n}{n}$ is equal to the number of carries in the addition $m + n$ done in base $p$.
\end{proof}

\begin{pro}\label{pro:23}
If $n$ is a power of $2$ and if all the digits of $n$ in base $3$ are equal to $0$ or $1$, then $\binom{2n}{n}$ is not a practical number.
\end{pro}
\begin{proof}
It follows by Lemma~\ref{lem:betap} that $\beta_2(n)=1$ and $\beta_3(n)=0$, that is, $\binom{2n}{n}$ is an integer of the form $12k\pm 2$. However, it is known that, other than $1$ and $2$, every practical number is divisible by $4$ or $6$, see \cite{MR0027799}.
\end{proof}

We will make use of the following result of probability theory.

\begin{lem}\label{lem:bindistr}
Let $X$ be a random variable following a binomial distribution with $j$ trials and probability of success $\alpha$.
Then
\begin{equation*}
\mathbb{P}[X \leq (\alpha - \varepsilon)j] \leq e^{-2\varepsilon^2 j}
\end{equation*}
for all $\varepsilon > 0$.
\end{lem}
\begin{proof}
See~\cite[Theorem~1]{MR0099733}.
\end{proof}

For each prime number $p$, let us define
\begin{equation*}
\alpha_p := \frac1{p} \left\lceil\frac{p - 1}{2}\right\rceil ,
\end{equation*}
so that $\alpha_p$ is the probability that a random digit in base $p$ is greater than $(p-1)/2$.

\begin{lem}\label{lem:estimatecentralbinomial}
Let $p$ be a prime number and fix $\varepsilon \in (0, \nicefrac1{2})$.
Then, for all $x \geq 1$, we have
\begin{equation*}
\beta_p(n) > (\alpha_p - \varepsilon) \frac{\log n}{\log p}
\end{equation*}
for all positive integers $n \leq x$ but at most $O(px^{1 - 2\varepsilon^2 / \log p})$ exceptions.
\end{lem}
\begin{proof}
For $x \geq 1$, let $k$ be the smallest integer such that $x < p^k$.
Clearly, we have
\begin{align}\label{equ:betapE1}
E(x) &:= \#\!\left\{n \leq x : \beta_p(n) \leq (\alpha_p - \varepsilon)\frac{\log n}{\log p} \right\} \nonumber\\
&\leq \sum_{j \,=\, 1}^k \#\big\{p^{j - 1} \leq n < p^j : \beta_p(n) \leq (\alpha_p - \varepsilon) j\big\} \nonumber\\
&\leq \sum_{j \,=\, 1}^k \#\big\{0 \leq n < p^j : \beta_p(n) \leq (\alpha_p - \varepsilon) j\big\} . 
\end{align}
Given an integer $j \geq 1$, let us for a moment consider $n$ as a random variable uniformly distributed in $\{0, \dots, p^j - 1\}$.
Then, the digits of $n$ in base $p$ are $j$ independent random variables uniformly distributed in $\{0,\dots, p - 1\}$.
Hence, as a consequence of Lemma~\ref{lem:betap}, we obtain that $\beta_p(n)$ follows a binomial distribution with $j$ trials and probability of success $\alpha_p$.
In turn, Lemma~\ref{lem:bindistr} yields
\begin{equation}\label{equ:betapE2}
\#\{0 \leq n < p^j : \beta_p(n) \leq (\alpha_p - \varepsilon) j\} \leq p^j e^{-2\varepsilon^2 j} .
\end{equation}
Therefore, putting together \eqref{equ:betapE1} and \eqref{equ:betapE2}, and using that $\varepsilon < \nicefrac1{2}$, we get
\begin{equation}\label{equ:Ellagain}
E(x) \le \sum_{j \,=\, 1}^k p^j e^{-2\varepsilon^2 j} \ll (p e^{-2\varepsilon^2})^k \leq (p e^{-2\varepsilon^2})^{\log x / \log p + 1} < p x^{1 - 2\varepsilon^2 / \log p} ,
\end{equation}
as desired.
\end{proof}

\begin{rmk}
The constant $\nicefrac1{2}$ in the statement of Lemma~\ref{lem:estimatecentralbinomial} has no particular importance, it is only needed to justify the $\ll$ in~\eqref{equ:Ellagain}.
Any other real number less than $(\tfrac1{2}\log 2)^{1/2}$ would be fine.
\end{rmk}

\section{Proof of Theorem~\ref{thm:row}}

Assume $x \geq 3$ sufficiently large, and put
\begin{equation*}
\varepsilon := \frac{\delta - \gamma}{\log \log x} + \frac{4 \log \log x}{\log x} \in (0, \nicefrac1{2}) .
\end{equation*}
Let $n$ be a positive integer.
By Lemma~\ref{lem:practicalp1ps} and by the definition of $\kappa(n)$, we know that $p_1 \cdots p_{\kappa(n)}$ is a practical number greater than or equal to $n$.
Since all the prime factors of $\binom{n}{k}$ are not exceeding $n$, Lemma~\ref{lem:practical2d} tell us that if $p_1 \cdots p_{\kappa(n)}$ divides $\binom{n}{k}$ then $\binom{n}{k}$ is practical.
Consequently, we have
\begin{equation*}
f(n) \leq \#\!\left\{0 \leq k \leq n : p_1 \cdots p_{\kappa(n)} \nmid \binom{n}{k}\right\} \leq \sum_{j \,=\, 1}^{\kappa(n)} T_{p_j} (n) .
\end{equation*}
Therefore, it follows from Lemma~\ref{lem:Tpupper} that
\begin{equation}\label{equ:gupper1}
f(n) < \sum_{j \,=\, 1}^{\kappa(n)} n^{\omega_{p_j} + \varepsilon} ,
\end{equation}
for all positive integers $n \leq x$ but at most
\begin{equation*}
\ll \sum_{j \,=\, 1}^{\kappa(x)} p_j^3 x^{1 - \varepsilon} \ll p_{\kappa(x)}^4 x^{1 - \varepsilon} \ll (\log x)^4 x^{1 - \varepsilon} = x^{1 - (\delta - \gamma) / \log \log x}
\end{equation*}
exceptions, where we also used Lemma~\ref{lem:asymptoticssxandpsx}.

Suppose that $n$ satisfies \eqref{equ:gupper1}.
Since $\omega_p$ is a monotone increasing function of $p$, we get that
\begin{equation}\label{equ:mfbound}
f(n) < \kappa(n) n^{\omega_{p_{\kappa(n)}} + \varepsilon} = n^{\omega_{p_{\kappa(n)}} + \log \kappa(n) / \log n + \varepsilon} .
\end{equation}

Moreover, for $n \gg_\gamma 1$ we have
\begin{equation}\label{equ:mbound1}
\omega_{p_{\kappa(n)}} < 1 - \frac{\log 2}{\log p_{\kappa(n)}} + \frac1{p_{\kappa(n)} \log p_{\kappa(n)}} < 1 - \frac{\log 2 - \gamma/4}{\log \log n} ,
\end{equation}
and
\begin{equation}\label{equ:mbound2}
\frac{\log \kappa(n)}{\log n} < \frac{\gamma / 4}{\log \log n} ,
\end{equation}
where we used Lemma~\ref{lem:asymptoticssxandpsx}.
Furthermore, since $n \leq x$, we have
\begin{equation}\label{equ:mbound3}
\varepsilon < \frac{\delta - \gamma/2}{\log \log n} .
\end{equation}
Consequently, putting together \eqref{equ:mbound1}, \eqref{equ:mbound2}, and \eqref{equ:mbound3}, we obtain
\begin{equation*}
\omega_{p_{\kappa(n)}} + \frac{\log \kappa(n)}{\log n} + \varepsilon < 1 - \frac{\log 2 - \delta}{\log \log n} ,
\end{equation*}
which, inserted into \eqref{equ:mfbound}, gives
\begin{equation*}
f(n) < n^{1 - (\log 2 - \delta)/\log \log n}
\end{equation*}
as desired.
The proof is complete.

\section{Proof of Corollary~\ref{cor:nk}}

Obviously, we can assume $\varepsilon < \tfrac1{2}\log 2$.
Put $\gamma := 2\varepsilon$ and $\delta := \tfrac{1}{2}\log 2 + \varepsilon$, so that $0 < \gamma < \delta < \log 2$.
For all $x \geq 3$, let $\mathcal{E}(x)$ be the set of exceptional $n \leq x$ of Theorem~\ref{thm:row}.
Then we have
\begin{align*}
\sum_{n \,\leq\, x} f(n) &= \sum_{n \,\notin\, \mathcal{E}(x)} f(n) + \sum_{n \,\in\, \mathcal{E}(x)} f(n) < \sum_{n \,\leq\, x} n^{1 - (\log 2 - \delta)/\log \log n} + \#\mathcal{E}(x)\,x \\
&\ll_\varepsilon x^{2 - (\log 2 - \delta) / \log \log x} + x^{2 - (\delta - \gamma) / \log \log x} \ll x^{2 - (\frac1{2}\log 2 - \varepsilon) / \log \log x} ,
\end{align*}
as claimed.

\section{Proof of Theorem~\ref{thm:central}}

For the sake of notation, put
\begin{equation*}
s := 16, \quad \eta_s := \frac{\sum_{i = 1}^s \alpha_{p_i} - 1}{\sum_{i = 1}^s \sqrt{\log p_i}}, \quad \varepsilon_j := \eta_s\sqrt{\log p_j} ,
\end{equation*}
for $j = 1,\dots,s$.
A computation shows that $\varepsilon_j \in (0, \nicefrac1{2})$ for $j=1,\dots,s$.

For $x \geq 1$, it follows from Lemma~\ref{lem:estimatecentralbinomial} that
\begin{equation}\label{equ:betapjlogpj}
\sum_{j \,=\, 1}^s \beta_{p_j}(n) \log p_j > \sum_{j \,=\, 1}^s (\alpha_{p_j} - \varepsilon_j)\log n = \log n ,
\end{equation}
for all positive integers $n \leq x$, but at most
\begin{equation*}
\ll \sum_{j \,=\, 1}^s p_j x^{1 - 2\varepsilon_j^2 / \log p_j} \ll x^{1 - 2\eta_s^2} < x^{0.88097}
\end{equation*}
exceptions.
Suppose that $n$ is a positive integer satisfying \eqref{equ:betapjlogpj}.
Then,
\begin{equation*}
d := \prod_{j \,=\, 1}^s p_j^{\beta_{p_j}(n)} > n .
\end{equation*}
Also, by Lemma~\ref{lem:practicalp1ps} we have that $d$ is a practical number, and by the definition of $\beta_{p_j}(n)$ we have that $d$ divides $\binom{2n}{n}$.
Moreover, since all the prime factors of $\binom{2n}{n}$ are not exceeding $2d$, Lemma~\ref{lem:practical2d} yields that $\binom{2n}{n}$ is practical.
The proof is complete.

\begin{rmk}
A comment is in order to explain the choice of the parameters $\varepsilon_j$ in the proof of Theorem~\ref{thm:central}.
Given a positive integer $s$, one could fix some prime numbers $q_1 <\cdots<q_s$ and some real numbers $\varepsilon_1,\ldots,\varepsilon_s \in (0,\nicefrac{1}{2})$ such that $q_1\cdots q_s$ is a practical number and $\sum_{j=1}^s (\alpha_{q_j}-\varepsilon_j)\ge 1$. 
Everything would proceed similarly, with an estimate of the number of exceptions given by 
\begin{equation*}
O\!\left(x^{\max\{1-2\varepsilon_1^2/\log q_1,\ldots,1-2\varepsilon_s^2/\log q_s\}}\right). 
\end{equation*}
To minimize the exponent of $x$, the optimal choice for $\varepsilon_j$ is 
\begin{equation*}
\varepsilon_j = \eta_s(q_1, \dots, q_s)\sqrt{\log q_j}, \quad \eta_s(q_1, \dots, q_s) := \frac{\sum_{i = 1}^s \alpha_{q_i} - 1}{\sum_{i = 1}^s \sqrt{\log q_j}} ,
\end{equation*}
for $j=1,\ldots,s$, which gives the estimate
\begin{equation*}
O\!\left(x^{1-2\eta_s(q_1,\ldots,q_s)^2}\right) .
\end{equation*}
Since $\alpha_p=\frac{1}{2}+O(\frac{1}{p})$ for each prime number $p$, we get that $\eta_s(q_1,\ldots,q_s)$ is maximized when $q_j=p_j$, for $j=1,\ldots,s$, and that $\eta_s(p_1, \dots, p_s) \to 0$ as $s \to +\infty$.
Lastly, some numeratical computations verify that the maximum of $\eta_s(p_1, \dots, p_s)$ is reached for $s = 16$.
\end{rmk}

\bibliographystyle{amsplain}

\end{document}